\documentclass[reqno,11pt]{amsart}
\usepackage{amsmath,amssymb,latexsym,soul,cite,mathrsfs}
\usepackage{color,enumitem,graphicx}
\usepackage[colorlinks=true,urlcolor=blue,
citecolor=red,linkcolor=blue,linktocpage,pdfpagelabels,
bookmarksnumbered,bookmarksopen]{hyperref}
\usepackage[english]{babel}
\usepackage[left=2.6cm,right=2.6cm,top=2.9cm,bottom=2.9cm]{geometry}
\usepackage{pgfplots}
\pgfplotsset{compat=1.15}
\usepackage{mathrsfs}
\usetikzlibrary{arrows}
\usepackage{geometry}
\usepackage{multirow}

\usepackage{float}
\usepackage{amssymb}
\usepackage{verbatim} 
\usepackage{amsmath}
\usepackage{amsthm}
\usepackage{epstopdf}
\usepackage{epsfig}
\usepackage[all]{xy}
\usepackage{enumerate}
\usepackage{hyperref}
\usepackage{graphicx}
\usepackage{caption}

\newtheorem{theorem}{Theorem}
\newtheorem{lemma}[theorem]{Lemma}
\newtheorem{corollary}[theorem]{Corollary}
\newtheorem{proposition}[theorem]{Proposition}

\newtheorem{definition}[theorem]{Definition}

\newtheorem*{theorem*}{Theorem}
\newtheorem*{lemma*}{Lemma}
\newtheorem*{corollary*}{Corollary}
\newtheorem*{proposition*}{Proposition}
\newtheorem*{remark*}{Remark}
\newtheorem*{definition*}{Definition}
\newtheorem*{conjecture*}{Conjecture}
\newtheorem*{claim*}{Claim}

\DeclareMathOperator{\vol}{\rm vol}
\DeclareMathOperator{\area}{\rm area}


\begin{document}
\title[]{A correspondence between genus one minimal Lawson surfaces of $\mathbb{S}^3(2)$ and area-minimizing unit vector fields on the antipodally punctured unit 2-sphere}
\author{Fabiano Brito$^1$}
\author{Jackeline Conrado$^2$}
\author{David L. Johnson$^3$}
\author{Giovanni Nunes$^4$}

\address{Centro de Matem\'atica, Computa\c c\~{a}o e Cogni\c c\~{a}o, 
Universidade Federal do ABC, Santo Andr\'{e} - SP, 09210-170, Brazil}
\email{brifabiano@gmail.com}

\address{Departamento de Geometria e Representação Gráfica, Instituto de Matem\'{a}tica e Estat\'{i}stica, 
Universidade do Estado do Rio de Janeiro, Rua São Francisco Xavier, 524 - Pavilhão Reitor João Lyra Filho, Rio de Janeiro - RJ, 20550-900, Brazil}
\email{jackeline.conrado@ime.uerj.br}

\address{Department of Mathematics, Lehigh University, 17 University Walk,
Bethlehem, PA 18105, USA.}
\email{david.johnson@lehigh.edu}

\address{Departamento de Matem\'atica e Estat\'istica, Instituto de F\'isica e Matem\'{a}tica,
Universidade Federal de Pelotas, Rua Gomes Carneiro 1, Centro - Pelotas, 96010-610, Brazil}
\email{giovanni.nunes@ufpel.edu.br}



\begin{abstract}
A correspondence is established between a class of minimal immersed surfaces of $\mathbb{S}^3(2)$ and area-minimizing unit vector fields defined on the antipodally punctured unit sphere $\mathbb{S}^2\backslash\{N,S\}$.
As a consequence, we establish a stability relation for Lawson cylinders in $\mathbb{S}^3(2)$.

\end{abstract}

\maketitle

\section{Introduction and main results}

The present work is a continuation of two previous papers [\textcolor{red}{3}] and [\textcolor{red}{2}], in which we studied area minimization of unit vector fields defined on the antipodally punctured unit sphere $\mathbb{S}^2\backslash\{N, S\}$ where $N=(0,0,1)$ and $S=(0,0,-1)$ are referred to as north and south poles of the unit 2-sphere $\mathbb{S}^2$, respectively. We define the area of a unit vector field $V$ on any surface $M$ as the area of the section defined by $V$ in $T^1M$, where $T^1M$ is endowed with the Sasaki metric and $V(M)\subset T^1M$ inherits the induced metric.

As a first step, we found a family of area-minimizing vector fields for each pair of Poincaré indices $\{I_V(N), I_V(S)\}=\{k, 2-k\}$ at the north and south poles of $\mathbb{S}^2$, respectively. Clearly, $I_V(N)+I_V(S)=2$. Precisely, we proved the following:

\begin{theorem}[Brito \textit{et al.}, \textcolor{red}{3}]\label{Thm:BCGN}
Let $V$ be a unit vector field defined on $M=\mathbb{S}^2\backslash\{N, S\}$. If $k = \max\{I_V(N), I_V(S)\}$, $k\neq0$, $k\neq2$ then 
\begin{eqnarray*}
    \area(V) \geq \pi L(\epsilon_k)
\end{eqnarray*} where $L(\epsilon_k)$ denote the length of the ellipse $\dfrac{x^2}{k^2}+\dfrac{y^2}{(k-2)^2}=1$ and $I_V(p)$ stands for the Poincaré index of $V$ around $p$.
   
\end{theorem}

It is also shown that these lower bounds are achieved by unit vector fields that make constant angles with meridians (just along them) and vary constantly in angle along each parallel, see Figure \ref{Fig:Area_Minimizing}. 
A precise definition of these vector fields on $\mathbb{S}^2 \backslash \{N, S\}$ is provided in the next section.

Naturally, the number of loops along each parallel depends on the Poincaré indices at the singularities $N$ and $S$. Let us denote by $V_k$ the area-minimizing unit vector field relative to $k=\max\{I_N(V), I_S(V)\}$ as defined in Theorem \ref{Thm:BCGN}.

In a second step, we geometrically describe the area-minimizing vector fields $V_k$ when $I_N(V)$ and $I_S(V)$ are even integers. If $k$ is even integer then the set $\overline{V_k \big(\mathbb{S}^2\backslash\{N,S\}\big)}$ corresponds to the image of a minimal Klein bottle in $T^1\mathbb{S}^2$. Specifically, we proved the following: 

\begin{theorem}[Brito \textit{et al.}, \textcolor{red}{2}] \label{Thm:GKlein}
Let $V_k$ be an area-minimizing unit vector field on $\mathbb{S}^2\backslash\{N,S\}$. If the Poincaré index around the singularity $N$ (or $S$) is $k \in 2\mathbb{Z}\backslash\{0, 2\}$, then the topological closure of $V_k \big(\mathbb{S}^2\backslash\{N,S\}\big)$ is a minimally immersed Klein bottle in $T^1\mathbb{S}^2$.
\end{theorem}

In this paper, we establish a correspondence between certain classes of genus $1$ immersed minimal surfaces in $\mathbb{S}^3$ introduced by Lawson in his classical paper [\textcolor{red}{9}] and area-minimizing unit vector fields in $T^1\mathbb{S}^2$ defined in [\textcolor{red}{3}].

It is worth noting that the geometry of the unit tangent bundle of $\mathbb{S}^2$ with the Sasaki metric is well known $(T^1\mathbb{S}^2, g^{Sas})$.  As shown by W. Klingenberg and T. Sasaki in [\textcolor{red}{8}], it is isometric to the real projective space $\left(\mathbb{R}{\rm{P}}^3(2), 4\overline{g}\right)$ which is obtained as the quotient of the sphere $\big(\mathbb{S}^3(2), g\big)$ of radius $2$, where $\overline{g}$ is the quotient metric. 

Let $\mathbf{SO}(3)$ denote the special orthogonal group equipped with the usual bi-invariant metric $\frac{1}{2}\left\langle\cdot ,\cdot \right\rangle$ given by
$\left\langle A,B \right\rangle = tr(A^{t}B)$, where $A,B \in \mathfrak{so}(3)$. 
 It is known that ($\mathbb{S}^3(2),g)$ is locally isometric to $\left(\mathbf{SO}(3), \frac{1}{2}\left\langle\cdot ,\cdot \right\rangle\right)$ by the Euler parametric transformation $\mathcal{E}: \mathbb{S}^3(2) \rightarrow \mathbf{SO}(3)$ given by quaternionic conjugation (see [\textcolor{red}{8}]). 
This map extends to $(T^1\mathbb{S}^2, g^{Sas})$ which is isometric to $\left(\mathbf{SO}(3), \frac{1}{2}\left\langle\cdot ,\cdot \right\rangle\right)$. 

This allows us to prove the following:

\begin{theorem}\label{Thm:Main1}

There are classes of compact minimal surface of Euler characteristic zero in $\mathbb{S}^3(2)$ which are locally isometric to the area-minimizing unit vector field $V_k$ on $\mathbb{S}^2\backslash\{N,S\}$.
\end{theorem}

\begin{theorem} \label{Thm:Main2}
    The classes of compact minimal surface of Euler characteristic zero in $\mathbb{S}^3(2)$ are exactly the Lawson surfaces $\tau_{n,n+1}$ which are minimally immersed Klein bottles and the cylindrical components of the Lawson surfaces $\tau_{2n+1,2n+3}$. 
\end{theorem}

For readers comfort, we explicitly reproduce the equations defining the Lawson surfaces $\tau$ in $\mathbb{S}^3$ and $\mathbb{S}^3(2)$.


As a consequence of Theorems \ref{Thm:Main1} and \ref{Thm:Main2}, we obtain the following corollary, which establishes a stability relation for Lawson cylinders.


\begin{corollary}\label{Cor:Lawson_Cylinders} If $C_{rs}^+$ and $C_{rs}^-$ are the cylindrical components of $\tau_{r,s}$ bounding two maximally-distant great circles of \,$\mathbb{S}^3(2)$, then $C_{rs}^+$ and $C_{rs}^-$ are locally area-minimizing among all immersed cylinders $C$ belonging to the family $\mathcal{F}$ of immersed cylinders in $\mathbb{S}^3(2)$ satisfying:
\begin{itemize}
    \item [(1)] The image of the interior of the cylinder $C$, $C\backslash\left\{ \partial C\right\} $, by the Euler parametric transformation is a unit vector field in $\mathbb{S}^2\backslash\{N,S\}$.
    \item [(2)] The boundary of $C$, $\partial C$, is the union of the two maximally-distant great circles of $\mathbb{S}^3(2)$.
\end{itemize}
\end{corollary}

In the following, we provide a visualization of the correspondence obtained between genus one minimal Lawson's surfaces and area-minimizing vector fields $V_k$ on \( \mathbb{S}^2 \setminus \{N, S\}\). 

The sets $C_{rs}^+$ and $C_{rs}^-$ are the cylindrical components of $\tau_{r,s}$, meaning that $C_{rs}^+ \cup C_{rs}^- = \tau_{r,s}$, where $r=2n+1$ and $s=2n+3$. Through the Euler transformation $\mathcal{E}: \mathbb{S}^3(2) \rightarrow \mathbf{SO}(3)$, there is a correspondence between $C_{rs}^+$ and $C_{rs}^-$ and the area-minimizing vector fields $V_k$ and $-V_k$, respectively, where $k$ is odd. In the case, $r=1$ and $s=3$, the Figures~\ref{Fig:CilindrosC+} and~\ref{Fig:CilindrosC-} illustrate this correspondence. Decompose \( C_{13}^+ \) into two pieces, each corresponding to the restriction of the area-minimizing vector field to the northern and southern hemispheres, respectively. The same decomposition is performed for \( C_{13}^- \). Consider \( C^+_{13}(N) \) and \( C^+_{13}(S) \), as well as \( C^-_{13}(N) \) and \( C^-_{13}(S) \), as the surfaces in \( \mathbb{R}^3 \) representing the pieces of \( C^+_{13} \) and \( C^-_{13} \) mentioned above, respectively.

\begin{figure}[H]
		\centering
		\includegraphics[height=15cm]{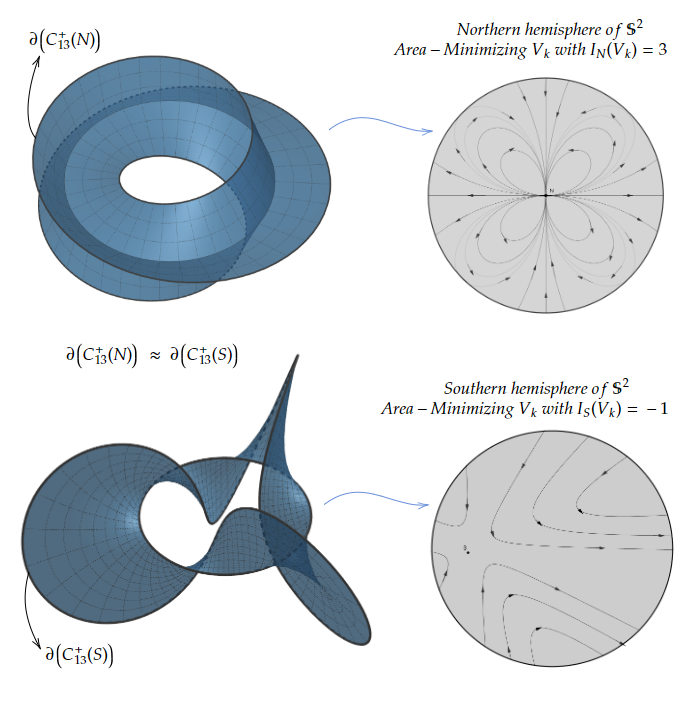}\\
        
		\caption{The correspondence between pieces of Lawson surfaces $\tau_{13}$ and area-minimizing unit vector fields $V_k$ on $\mathbb{S}^2\backslash\{N,S\}$.}
		\label{Fig:CilindrosC+}
	\end{figure}

\begin{figure}[H]
		\centering
		\includegraphics[height=19cm]{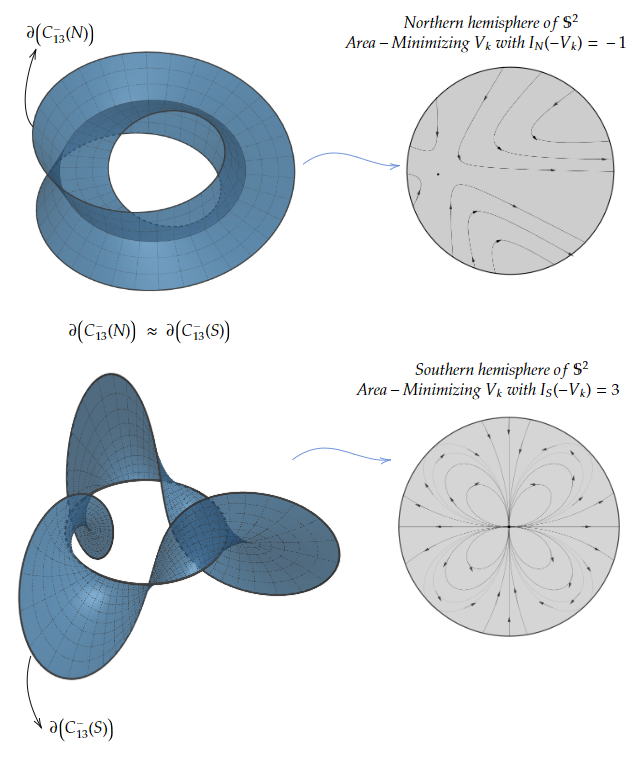}
		\caption{The correspondence between pieces of Lawson surfaces $\tau_{13}$ and area-minimizing unit vector fields $-V_k$ on $\mathbb{S}^2\backslash\{N,S\}$.}
		\label{Fig:CilindrosC-}
	\end{figure}

The visualization of the correspondence between the area-minimizing vector field $V_k$, and the Lawson surfaces $\tau_{n,n+1}$ is given by Figure \ref{Fig:FaixasMoebius}, for specific case $n=1$. Let \( F(N) \) and \( F(S) \) denote the Moebius strips corresponding to the restriction of the area-minimizing vector field \( V_k \) to the northern and southern hemispheres, respectively.

\begin{figure}[H]
		\centering
		\includegraphics[height=16cm]{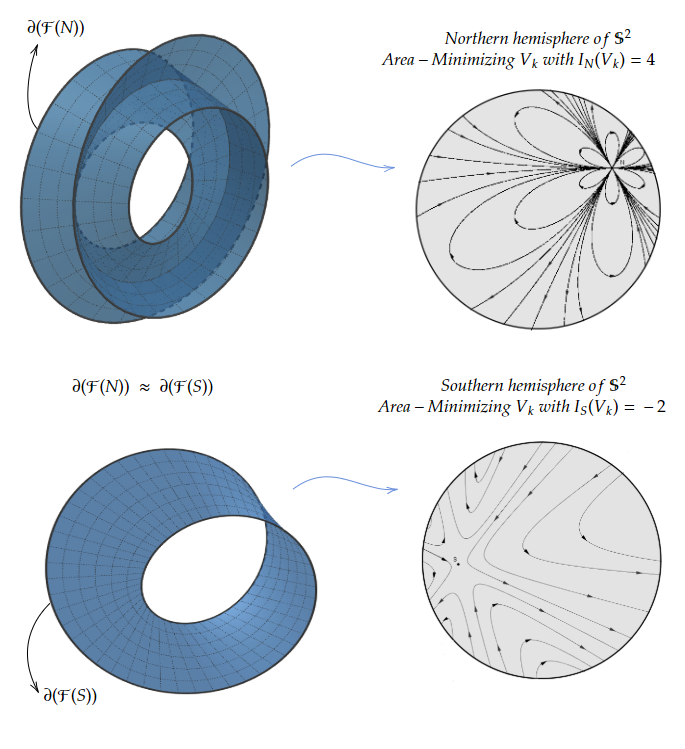}
		\caption{The correspondence between pieces of Lawson surfaces $\tau_{12}$ and area-minimizing unit vector fields $V_k$ on $\mathbb{S}^2\backslash\{N,S\}$.}
		\label{Fig:FaixasMoebius}
	\end{figure}

\section{Differential Geometric Preliminaries}

\subsection{Area-minimizing unit vector fields on antipodally punctured unit 2-sphere}

Let $M$ be a closed oriented Riemannian manifold and $V$ a unit vector field on $M$. Consider the unit tangent bundle $T^1M$ equipped with the Sasaki metric. The \textit{\textbf{volume of a unit vector field $V$}} is defined (see [\textcolor{red}{6}]) as the volume of the submanifold $V(M)$, the image of the immersion $V: M \rightarrow T^1M$, 
\begin{equation*}\label{Eq:volume}
\vol(V):=\vol(V(M)).
\end{equation*}

\noindent Let $\{e_1,\cdots ,e_n\}$ be an orthonormal local frame on $M$ and denote by $\nu_M$ the volume form of $M$ written with respect to it. The formula for the volume of the unit vector field $V$ is given by
\begin{eqnarray}\label{Eq:volumegeral}\nonumber
\vol( V) & = &\int_{M}{\sqrt{\det(\rm{I}+(\nabla V)(\nabla V)^*)}\nu_M}\\
& = & \int_{M} \bigg(1 + \sum_{j}^{n}||\nabla_{e_j}V||^2 + \sum_{j_1 < j_2}||\nabla_{e_{j_1}}V \wedge \nabla_{e_{j_2}}V||^2 + \cdots \\\nonumber
& &+ \cdots + \sum_{j_1 < \cdots < j_{n-1}}||\nabla_{e_{j_1}}V \wedge \cdots\wedge\nabla_{e_{j_{n-1}}}V||^2\bigg)^{\frac{1}{2}},
\end{eqnarray}
where $\rm{I}$ is the identity, and $\nabla V$ is considered as an endomorphism of the tangent space with adjoint operator $(\nabla V)^*$, see [\textcolor{red}{7}]. If a unit vector field is parallel, i.e. $\nabla V = 0$, equation \eqref{Eq:volumegeral} implies there is a trivial minimum, $\vol(V)=\vol(M).$


\noindent If $M$ is a surface and  $V$ a unit vector field on $M$. The \textit{\textbf{area of a unit vector field $V$}} on $M$ is defined as the area of the surface $V(M)$ in $T^1M$, and it coincides with the volume of a unit vector field $V$
\begin{align*}
    \area(V) := \area(V(M)) = \vol(V).
\end{align*}

\noindent It is well known that there is no globally defined vector field on $\mathbb{S}^2$, so the infimum of the area is achieved only when there is at least one singularity. The Pontryagin vector field is a unit vector field with one singularity obtained by parallel translating a given vector along any great circle passing through a given point. In [\textcolor{red}{1}], Borrelli and Gil-Medrano showed that Pontryagin vector fields of the unit 2-sphere with one singularity are area-minimizing. 

\begin{figure}[H]
		\centering
		\includegraphics[height=6cm]{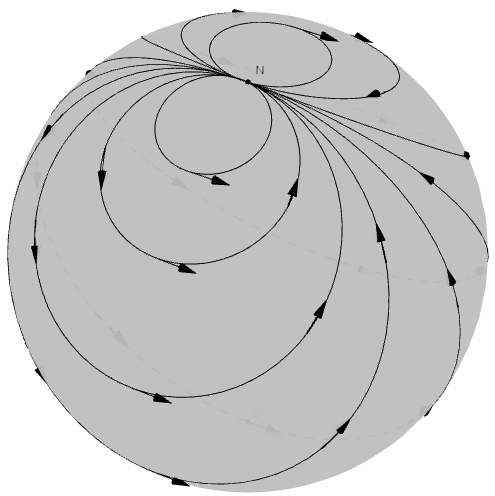}
		\caption{The Pontryagin vector field on $\mathbb{S}^2\backslash\{N\}$ with Poincaré index $2$ at the singularity $N$}
		\label{Fig:V_2}
	\end{figure}

As mentioned in Theorem \ref{Thm:BCGN}, we established sharp lower bounds for the total area of unit vector fields on the antipodally punctured unit 2-sphere, and these values depend on the indices of the singularities. Furthermore, we exhibited a family of the unit vector fields $V_k$ on  $\mathbb{S}^2\backslash\{N,S\}$ that are area-minimizing within each index class $k$. We shall present their precise definition.

Let $\mathbb{S}^{2}\backslash\{N,S\}$ be the Euclidean sphere with two antipodal points, $N$ and $S$, removed. Denote by $g$ the usual metric of $\mathbb{S}^2$ induced from $\mathbb{R}^3$, and by $\nabla$ the Levi-Civita connection associated to $g$. Consider the oriented orthonormal frame $\left\{ e_1, e_2 \right\}$ on $\mathbb{S}^2\backslash\left\{ N,S\right\}$ such that $e_1$ is tangent to the parallels and $e_2$ to the meridians.
Let $k$ be an integer number and define the \textit{\textbf{angle function}} as 
\begin{eqnarray} \nonumber
	\theta_k: \mathbb{S}^2 \backslash \left\{N, S \right\} &\longrightarrow& \mathbb{R}\\
	p &\longmapsto& \theta_k(p) = (k - 1)t + \frac{\pi}{2},\label{Def:angle_function}
\end{eqnarray}
where $t \in [0,2\pi)$ is the \textit{longitude} coordinate of $p$ in $\mathbb{S}^2\backslash\left\{ N,S\right\}$. 

\begin{definition} \label{def:Vks} \index{area-minimizing vector field on $\mathbb{S}^2\backslash\{\pm p\}$}
For $k \in \mathbb{Z}$, define the unit vector field $V_k$ as
\begin{eqnarray} \label{Eq:Vks}
V_k(p) = \cos\left(\theta_k(p)\right)e_1(p) + \sin\left(\theta_k(p)\right)e_2(p),
\end{eqnarray}
where $p \in \mathbb{S}^2\backslash\left\{ N,S\right\}$, $\theta_k$ is the angle function and $\{e_1, e_2\}$ is the oriented orthonormal frame on $\mathbb{S}^2\backslash\left\{ N,S\right\}$ described above.
\end{definition}

Definition~\ref{def:Vks} coincides with the one in [\textcolor{red}{3}].
In either case, the defined vector fields differ from the Pontryagin fields in that the latter are vector fields with only one singularity [\textcolor{red}{1}]. The unit vector fields $V_k$ are parallels along meridians and wind $k-1$ times around the parallel at a constant angle speed with respect to the oriented orthonormal frame $\left\{e_1, e_2 \right\}$, (see [\textcolor{red}{3}, \textcolor{red}{4}]). 
Figure \ref{Fig:Area_Minimizing} shows the vector field $V_k$ on $\mathbb{S}^2\backslash\{N,S\}$ (in blue), where $S_{\alpha}$ and $S_0$ denote the parallel $\alpha$ and the Equator, respectively. 

\begin{figure}[H]
		\centering
		\includegraphics[height=7cm]{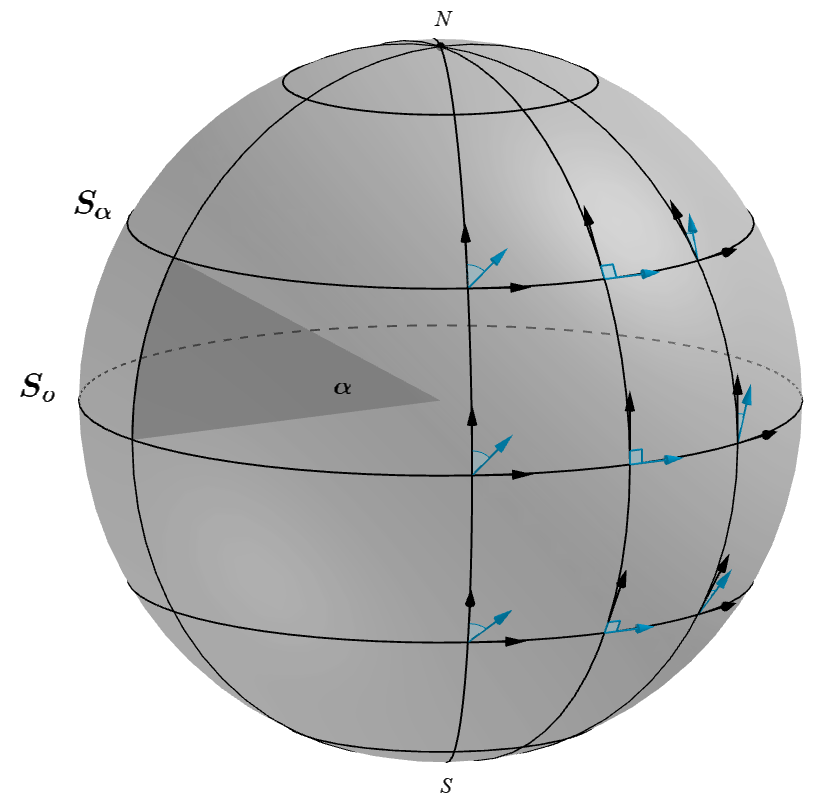}
		\vspace{0.5cm}
		\caption{Area-minimizing unit vector field $V_{k}$ on $\mathbb{S}^2\backslash\{N,S\}$}
		\label{Fig:Area_Minimizing}
	\end{figure}

Consider another oriented orthonormal local frame $\{V_k, V_k^\perp \}$ on $\mathbb{S}^2\backslash\left\{ N,S\right\}$ compatible with the orientation of $\{e_1, e_2\}$ and its dual basis $\{\omega_1, \omega_2\}$. In dimension $2$, the volume of $V_k$ is given by
\begin{eqnarray}
    \vol \left(V_k\right) = \int_{\mathbb{S}^2} \sqrt{1+ \gamma^2 + \delta^2}\,\,\nu,
\end{eqnarray}
where $\nu$ is the volume form, $\gamma = g\left(\nabla_{_{V_k}} V_k \,, V_k^\perp\right)$ and $\delta = g\left(\nabla_{_{V_k^\perp}} V_k^\perp \, , V_k\right)$ are the geodesic curvatures associated to $V_k$ and $V_k^\perp$, respectively.

\noindent Let $\theta_k \in [0, 2\pi)$ be the oriented angle from $e_2$ to $V_k$ defined by \eqref{Def:angle_function} and $\alpha \in \left[ -\frac{\pi}{2}, \frac{\pi}{2} \right]$. If $V_k = \cos\left(\theta_k\right)e_1 + \sin\left(\theta_k\right)e_2$ and $V_k^\perp = -\sin\left(\theta_k\right)e_1 + \cos\left(\theta_k\right)e_2$, then
\begin{equation}\label{curva_geode_tangente}
\gamma = \cos \left(\theta_k\right) \big(\tan \alpha + d\theta_k(e_1)\big) + \sin\left(\theta_k\right) d\theta_k(e_2), 
\end{equation}
\begin{equation}\label{curva_geode_ortogonal}
\delta = \sin (\theta_k) \big(\tan \alpha + d\theta_k(e_1)\big) - \cos (\theta_k) d\theta_k(e_2),
\end{equation}
and 
\begin{eqnarray}\label{Eq:Integrando_Volume}
    1+\gamma^2+\delta^2 = 1 + \big( \tan \alpha + d\theta_k(e_1)\big)^2 + d\theta_k(e_2)^2.
\end{eqnarray}

\noindent The Lemma in [\textcolor{red}{3}] provides a straightforward computation from \eqref{curva_geode_tangente} to \eqref{Eq:Integrando_Volume}.

\noindent Let $S^1_\alpha$ be the parallel of $\mathbb{S}^2$ at latitude $\alpha \in \left[ -\frac{\pi}{2}, \frac{\pi}{2} \right]$ and $S^1_\beta$ be the meridian of $\mathbb{S}^2$ at longitude $\beta \in (0, 2\pi)$. 
 The equation \eqref{Eq:Integrando_Volume} allows us to rewrite the volume functional as an integral depending on the latitude $\alpha$ and the derivatives of $\theta_k$:
\begin{eqnarray}
    \vol\left(V_k\right) = \int_{\mathbb{S}^2} \sqrt{1+  \big( \tan \alpha + d\theta_k(e_1)\big)^2 + d\theta_k(e_2)^2}\,\nu,
\end{eqnarray}

\noindent Notice that $\theta_k$ has constant variation along the parallel $\cos \alpha$, with $\alpha \in \left[ -\frac{\pi}{2}, \frac{\pi}{2} \right]$ constant, in that way
\begin{equation}\label{Eq:variacao_constante}
d\theta_k(e_1) = \frac{k-1}{\cos \alpha}\quad\text{and}\quad d\theta_k(e_2) = 0.
\end{equation}

\noindent Definition \ref{def:Vks} allows us the family of unit vector fields $V_k$ that satisfies \eqref{Eq:variacao_constante} and the next computation obtained from Proposition 8 in [\textcolor{red}{3}], shows us that for each index class $k=\max\{I_N(V), I_S(V)\}$, the unit vector field $V_k$ is area-minimizing on $\mathbb{S}^2\backslash\{N,S\}$.  

\noindent Since  $k > 2$, we consider the ellipse $\varepsilon_{k}$ parametrized by
$\mu(t) = (k \cos t, (k-2)\sin t)$, whose length is
\begin{eqnarray}\label{Eq: Comprimento_Elipse}
    L(\epsilon_k) = 4\int_{0}^{\frac{\pi}{2}}\sqrt{ (k-2)^2 + 4(k-1)\sin^2 (t)} dt.
\end{eqnarray}
If the family of unit vector fields $V_k$ that satisfies \eqref{Eq:variacao_constante}, we have
\begin{align*}
\vol(V_k) 
    & =  \int_{\mathbb{S}^2} \sqrt{1+  \big( \tan \alpha + d\theta_k(e_1)\big)^2 + d\theta_k(e_2)^2}\, \nu\\
    & = \int_{\mathbb{S}^2}{\sqrt{1+ \left( \frac{\sin \alpha + k-1}{\cos\alpha}\right)^2}}\,\nu \\
    & =  \lim_{\alpha_0 \to -\frac{\pi}{2}} \int_{\alpha_0}^{\frac{\pi}{2}}\int_{0}^{2\pi}{\sqrt{ \frac{1+ (k-1)^2 +2(k-1)\sin\left(\frac{\alpha}{2} + \frac{\pi}{4}\right)}{\cos^2\alpha}}}\,\cos\alpha \,d\beta d\alpha \\
    & =  2\pi\int_{-\frac{\pi}{2}}^{\frac{\pi}{2}}\sqrt{1+ (k-1)^2 +2(k-1)\sin \left(\frac{\alpha}{2} + \frac{\pi}{4}\right)}\, d\alpha.
\end{align*}
Letting $t := \frac{\alpha}{2} + \frac{\pi}{4}$ and $L(\epsilon_k)$ given by \eqref{Eq: Comprimento_Elipse}, we obtain
\begin{align*}
\vol(V_k)  & = 4\pi\int_{0}^{\frac{\pi}{2}}\sqrt{(k-2)^2 + 4(k-1)\sin^2t}\,dt \\ 
                 & = \pi L(\varepsilon_k),
\end{align*}

\noindent On the other hand, if $\vol(V_k ) = \pi L(\epsilon_k)$.
\begin{align*}
\vol(V_k) & = \int_{\mathbb{S}^2} \sqrt{1+  \big( \tan \alpha + d\theta_k(e_1)\big)^2 + d\theta_k(e_2)^2}\, \nu\\
    & \geq \int_{\mathbb{S}^2} \sqrt{1+  \big( \tan \alpha + d\theta_k(e_1)\big)^2}\, \nu\\
    & \geq \int_{-\frac{\pi}{2}}^{\frac{\pi}{2}}{|\cos \varphi + \sin\varphi \big( \tan \alpha + d\theta_k(e_1)\big)|}\, d\alpha \\
    & = \pi L(\varepsilon_k),
\end{align*}
then $d\theta_k(e_2) = 0$ and $\cos\varphi(\tan \alpha + d\theta_k(e_1)) = \sin\varphi$, where $\varphi \in \mathbb{R}$. The last inequality is obtained by the general inequality, $\sqrt{a^2 + b^2} \geq |a\cos \varphi + b \sin \varphi|$, for any $a$, $b$, $\varphi \in \mathbb{R}$, taking $a=1$ and $b=\tan \alpha + d\theta_k(e_1)$. We conclude that, $d\theta_k(e_1) =  \tan\varphi - \tan\alpha$ and $\varphi = \varphi(\alpha) = \arctan\left(\tan \alpha + d \theta_k(e_1)\right)$, which implies $d\theta_k(e_1)= \dfrac{k-1}{\cos \alpha}$.



\subsection{Minimally immersed Klein bottles in $T^1\mathbb{S}^2$ arising from area-minimizing unit vector fields on $\mathbb{S}^2\backslash \{N,S \}$} \label{GK}

If a unit vector field minimizes the area  among unit vector fields its image is a minimal surface of the unit tangent bundle [\textcolor{red}{5}]. 
Therefore, for each area-minimizing unit vector fields $V_{k, 2-k}$ on $\mathbb{S}^2\backslash\{N,S\}$ 
provide us a minimal surface in $T^1\mathbb{S}^2$. Surprisingly, when the Poincaré index $k$ is even positive and is neither zero nor $2$, these minimal surfaces are minimally immersed Klein bottles in $T^1\mathbb{S}^2$ as mentioned in Theorem \ref{Thm:GKlein}.

The minimal immersed Klein bottle obtained in $T^1\mathbb{S}^2(1)$ is given by gluing two images of the Moebius strip along their respective boundaries. It is possible to make this collage because the Poincaré index $k$ is even. In this way, the technique used cannot be applied to get similar results to odd Poincaré indices. For more details, see [\textcolor{red}{2}, \textcolor{red}{4}].



\begin{remark*}
    Antonio Ros proved that the Klein bottle cannot be embedded in $T^1\mathbb{S}^2$ (see [\textcolor{red}{10}]). 
\end{remark*}



\subsection{The Euler parametric transformation of $\mathbf{SO}(3)$}\label{subsec:Euler}
In 1975, W. Klingenberg and S. Sasaki showed that the unit tangent bundle $T^1\mathbb{S}^2$ is isometric to the projective space $\mathbb{R}\rm{P}^3(2)$ via the Euler parametric transformation of $\mathbf{SO}(3)$, as defined in \eqref{Matrix_Euler_param}. The complete proof can be found in [\textcolor{red}{1}, \textcolor{red}{8}].
At this point, we recall the maps and isometries that will be needed throughout this paper. Let $\mathbf{SO}(3)$ be the special orthogonal group equipped with the usual bi-invariant metric $\frac{1}{2}\left\langle\cdot ,\cdot \right\rangle$ given by
$\left\langle A,B \right\rangle = tr(A^{t}B)$,
where $A,B \in \mathfrak{so}(3):= T_I SO(3)$, and $I$ is the identity matrix. 
The map $\psi(p, V_p) = \left(p,\,V_p, \, \vec{p} \wedge V_p\right)$ is an isometry from $\big(T^1\mathbb{S}^2, g^{Sas}\big)$ onto $\left(\mathbf{SO}(3), \frac{1}{2}\left\langle\cdot ,\cdot \right\rangle\right)$, where $p \in \mathbb{S}^2$ and $\vec{p}$ is its corresponding position vector in $\mathbb{S}^2$, and $\wedge$ denotes the standard cross product between the vectors in $\mathbb{R}^3$. Indeed, for an element $(p,V_p) \in T^1\mathbb{S}^2$, consider the isometry of $\mathbb{R}^3$ denoted by $\psi(p,V_p)$ and defined as below:
\begin{align*}
    \psi(p,V_p)(e_1) = p, \hspace{1cm} \psi(p,V_p)(e_2) = V_p, \hspace{1cm} \psi(p,V_p)(e_3) = \vec{p}\wedge V_p,
\end{align*}
where $\{ e_1, e_2, e_3 \}$ is the usual orthonormal basis of $\mathbb{R}^3$.

\noindent 
If $(p,V_p) \in T^1\mathbb{S}^2$ and $\xi \in T_{(p,V_p)} T^1\mathbb{S}^2$, consider $\mu(t) = (p(t), V_p(t))$ be a curve in $T^1\mathbb{S}^2$, we can assume without loss of generality that, $p(0)= p = e_1$, $V_p(0) = V_p = e_2$ and $\mu(0) = \xi$. Then $d\psi_{(e_1,e_2)}\big(\xi\big) \in \mathfrak{so}(3)$ and
\begin{align*}
    d\psi_{(e_1,e_2)}\big(\xi\big) = \bigg(p'(0), \,\, V'_p(0), \,\, p'(0)\wedge e_2 + e_1 \wedge V'_p(0)\bigg).
\end{align*}
By straightforward computation one obtains
\begin{align*}
    \frac{1}{2}\left\langle d\psi_{(e_1,e_2)}\big(\xi\big) ,d\psi_{(e_1,e_2)}\big(\xi\big) \right\rangle = |p'_2(0)|^2 + |p'_3(0)|^2 + |V'_3(0)|^2, 
\end{align*}
and
\begin{align*}
    g^{Sas}(\xi, \xi)= g\big(p'(0), p'(0)\big) + g \left( \frac{\nabla V}{dt}(0), \frac{\nabla V}{dt}(0) \right) = |p'_2(0)|^2 + |p'_3(0)|^2 + |V'_3(0)|^2. 
\end{align*}
It remains to show that $\mathbb{R}\rm{P}^3(2)$ is isometric to $\mathbf{SO}(3)$. The key to proving this is to rewrite the Euler parametric transformation through an action by conjugation of subgroups, which we will explain below.

\noindent
Let $g$ denote the standard metric on $\mathbb{R}^n$. Consider the \textit{\textbf{Euler parametric transformation}} 
\begin{align}
    \label{Matrix_Euler_param} \index{Euler parametric transformation of $\mathbf{SO}(3)$}\nonumber 
    \Phi: \left( \mathbb{S}^3(2),g \right) &\longrightarrow \left(\mathbf{SO}(3), \frac{1}{2}\left\langle\cdot ,\cdot \right\rangle\right)\\ 
    (x_1,x_2,x_3,x_4) &\longmapsto
    \frac{1}{4}\begin{pmatrix}
    x_1^2+x_2^2-x_3^2-x_4^2 & 2x_1x_4 +2x_2x_3 & -2x_1x_3+2x_2x_4\\
    -2x_1x_4 + 2x_2x_3 & x_1^2-x_2^2+x_3^2-x_4^2 & 2x_1x_2+2x_3x_4\\
    2x_1x_3+2x_2x_4 & -2x_1x_2+2x_3x_4 & x_1^2-x_2^2-x_3^2+x_4^2
    \end{pmatrix}.
\end{align}
We identify $\mathbb{R}^4$ with the quaternions $\mathbb{H}$, and $\mathbb{R}^3$ with the imaginary quaternions $\rm{Im}\mathbb{H}$. Thus, we have $\mathbf{SO}(3) = \mathbf{SO}(\rm{Im}\mathbb{H})$. Since $\mathbb{S}^3$ is a subgroup of $\mathbb{H}$ and $\mathbb{S}^3$ acts on $\mathbf{SO}(\rm{Im}\mathbb{H})$ by conjugation, the Euler parametric transformation $\Phi$ defined in \eqref{Matrix_Euler_param} can be expressed by 
\begin{eqnarray*}
    \Phi(q)(u)=q^{-1}uq, \hspace{1cm} q \in \mathbb{S}^3.
\end{eqnarray*}
The map $\Phi$ is a double covering of $\mathbf{SO}(3)$ and induces a diffeomorphism $\overline{\Phi}$ between $\mathbb{R}{\rm{P}}^3(2)$ and $\left(\mathbf{SO}(3), \frac{1}{2}\left\langle\cdot ,\cdot \right\rangle\right)$, where $\overline{g}$ is the quotient metric. The differential of $\Phi$ at $e=1 \in \mathbb{S}^3$ is a linear map from $T_e\mathbb{S}^3 = \rm{Im}\mathbb{H}$ to the Lie algebra $\mathfrak{so}(\rm{Im}\mathbb{H})$ is given by $d \Phi_e(X)(u) =  Xu-uX$, for all $X \in T_e\mathbb{S}^3$ and $u \in \rm{Im}\mathbb{H}$. Since $X$ and $u$ are purely imaginary, the product $Xu$ is again imaginary and can be identified with the usual wedge product in $\mathbb{R}^3 \cong \rm{Im}\mathbb{H}$. Thus,
\begin{align}
    d\Phi_e(X) = \begin{pmatrix}
    0 & -X_3 & X_2\\
    X_3 & 0 & -X_1\\
    -X_2 & X_1 & 0
    \end{pmatrix}.
\end{align}
By straightforward computation one obtains
\begin{align*}
    \left\langle d\Phi_e(X)) , d\Phi_e(X)) \right\rangle = 8 || X ||^2 \hspace{1cm}\mbox{and}\hspace{1cm} \Phi^*\left\langle \, \, , \, \,  \right\rangle = 8g.
\end{align*}
The projective space can be defined as $\mathbb{R}\rm{P}^3(2) = \mathbb{S}^3(2)/\mathbb{Z}_2$. The map $\Phi$ induces an isometry $\overline{\Phi}$ between $\left(\mathbb{R}{\rm{P}}^3(2), 4\overline{g}\right)$ and $\left(\mathbf{SO}(3), \frac{1}{2}\left\langle\cdot ,\cdot \right\rangle\right)$, where $\overline{g}$ is the quotient metric.  
The composition $\varphi = \psi^{-1}\circ \overline\Phi$ is then an isometry between $\left(\mathbb{R}{\rm{P}}^3(2), 4\overline{g}\right)$ and $\big( T^1\mathbb{S}^2, g^{Sas}\big)$. Moreover, by definition $\pi(\varphi(e^{i\theta} q)) = \Phi(e^{i\theta} q)(i) = \pi(\varphi(q))$, which means $\varphi$ preserves the natural structures of $\mathbb{S}^1$-bundle over $\mathbb{S}^2$ of both spaces $\mathbb{S}^3$ and $T^1\mathbb{S}^2$.
 \begin{eqnarray}
    \label{Eq:Isometria_Tangente_RP3}
    \xymatrix{
    \big( \mathbb{R}\rm{P}^3(2), 4\overline{g} \big) 
    \ar[rr]^{\psi^{-1}\circ \overline{\Phi}}
    \ar[rd]_{\overline\Phi} 
    & & \big(T^1\mathbb{S}^2, g^{Sas}\big)  \\
    & \big( \mathbf{SO}(3), \frac{1}{2}\left\langle \cdot , \cdot \right \rangle \big) \ar[ru]_{\psi^{-1}}
    }
    \end{eqnarray}
To simplify the notation, we shall denote by $\mathcal{E}: \mathbb{S}^3(2) \rightarrow \mathbf{SO}(3)$  the Euler transformation. 


\subsection{Lawson $\tau$ surfaces}

Let $\mathbb{S}^3 = \{(z,w) \in \mathbb{C}^2;\, |z|^2 + |w|^2 = 1\}$ be a unit $3$-sphere. 
\begin{definition}
    A Lawson surface $\tau_{n,m}$ in $\mathbb{S}^3$  is defined as the image of the doubly periodic immersion $\varphi_{n,m}: \mathbb{R}^2 \hookrightarrow \mathbb{S}^3 \subset \mathbb{R}^4$, given by the explicit formula:
    \begin{eqnarray}\label{Eq:tau_surface}
    \varphi_{n,m}(x, y) = \big(\cos(mx)\cos(y),\,\sin(mx) \cos(y),\, \cos(nx) \sin(y),\, \sin(nx) \sin(y)\big),\,\, n,m \in \mathbb{Z}.     
    \end{eqnarray}
\end{definition}

\noindent Theorem $1$ in [\textcolor{red}{9}] shows that $\tau_{n,m}$ is a compact, non-singular minimal submanifold with Euler characteristic zero. Observe that, from equation given by \eqref{Eq:tau_surface}, $\tau_{n,m} = \tau_{n',m'}$ (up to congruences) if and only if there exists an integer $p$ such that $m \cong m' (\mbox{mod} \, p)$ and $n \cong n'(\mbox{mod} \, p)$. Hence, we have a countable family of compact surfaces with the following properties.

\begin{theorem}[Lawson, \textcolor{red}{9}] \label{Thm:Propriedades_Lawsons}
To each unordered pair of integers $m, n$ with
$(m, n) = 1$ there corresponds a distinct, compact minimal surface $\tau_{n,m}$ of Euler characteristic zero in $\mathbb{S}^3$ given by \eqref{Eq:tau_surface}. Moreover,
\begin{itemize}
    \item [(a)] $\tau_{n,m}$ is non-orientable (an immersed Klein bottle) if and only if $\,2/mn$.
    \item [(b)] $\tau_{n,m}$ is real algebraic of degree $n + m$ and satisfies the equation
    \begin{eqnarray*}
        {\rm Im} \{z^n\bar{w}^m\} = 0.
    \end{eqnarray*}
    \item [(c)] Each $\tau_{n,m}$ admits a distinct one-parameter group of self-congruences. 
    \item [(d)] Each $\tau_{n,m}$ is geodesically ruled.
    \item [(e)] Area ($\tau_{n,m}$) $\geq 2\pi^2 \min\{n,m\}$. 
    \item [(f)]  $\tau_{1,1}$ is the Clifford torus and the only surface $\tau_{n,m}$ without self-intersections.
\end{itemize}
\end{theorem}

\noindent The classes of compact minimal surfaces $\tau_{n,m}$ of Euler characteristic zero in $\mathbb{S}^3$, given by $\tau_{n,n+1}$ and $\tau_{r,s}$, where $r=2n+1$ and $s=2n+3$, are classes of ruled minimal surfaces of genus $1$ in $\mathbb{S}^3$ given by the explicit formulas:
\begin{eqnarray*}
    \varphi_{n,n+1}(x, y)&=& 
    \bigg(\cos\big((n+1)x\big)\cos(y),\,
    \sin\big((n+1)x\big) \cos(y),\, 
    \cos(nx) \sin(y),\, 
    \sin(nx) \sin(y)\bigg),       
\end{eqnarray*}
\begin{eqnarray*}
    \varphi_{r,s}(x, y) = 
    \bigg(\cos\big((2n+1)x\big)\cos(y),\,
    \sin\big((2n+1)x\big) \cos(y),\, 
    \cos((2n+3)x) \sin(y),\, 
    \sin((2n+3)x) \sin(y)\bigg).        
\end{eqnarray*}

\noindent Observe that the Clifford torus $\tau_{1,1}$ is congruent to $\tau_{-1,1}$ mod $2$.


\section{Proofs of Theorems}
Since the unit tangent bundle $T^1\mathbb{S}^2$ is isometric to the projective space $\mathbb{R}\rm{P}^3(2)= \mathbb{S}^3(2)/\mathbb{Z}_2$, as mentioned in \eqref{Eq:Isometria_Tangente_RP3}, so at this point we are interested in the family of minimal surfaces $\tau_{n,m}$ that lie on the $3$-sphere of radius $2$. 

\begin{definition}\label{Def:Lawson_3-sphere(2)}
    A Lawson surface $\tau_{n,m}$ in $\mathbb{S}^3(2)$  is defined as the image of the doubly periodic immersion $\varphi_{n,m}: \mathbb{R}^2 \hookrightarrow \mathbb{S}^3(2) \subset \mathbb{R}^4$, given by the explicit formula
    \begin{eqnarray}\label{Eq:tau_Lawson_3-sphere(2)}
    \varphi_{n,m}(x, y) = \big(2\cos(nx)\cos(y),\,2\sin(nx) \cos(y),\, 2\cos(mx) \sin(y),\, 2
    \sin(mx) \sin(y)\big),\,\, n,m \in \mathbb{Z}.     
    \end{eqnarray}
\end{definition}

\noindent To each unordered pair of integers $m, n$ with
$(m, n) = 1$, there corresponds a distinct compact minimal Lawson surface $\tau_{n,m}$ of Euler characteristic zero in $\mathbb{S}^3(2)$, given by \eqref{Eq:tau_Lawson_3-sphere(2)} with the equivalent properties stated in Theorem \ref{Thm:Propriedades_Lawsons}, modulo spherical homothety. 

To each family of minimal Lawson surfaces $\tau_{n,m}$ in $\mathbb{S}^3(2)$, there corresponds a family of minimal surfaces $\Sigma_{n,m}$ in $T^1\mathbb{S}^2$ via the Euler transformation $\mathcal{E}: \mathbb{S}^3(2) \rightarrow \mathbf{SO}(3)$. In what follows, for each family of minimal Lawson surfaces $\tau_{n,m}$ in $\mathbb{S}^3(2)$, we compute the corresponding matrix $M_{n,m}$ in $\mathbf{SO}(3)$. 
From \eqref{Matrix_Euler_param} and \eqref{Eq:tau_Lawson_3-sphere(2)} above, we have 
\begin{eqnarray}
   \Phi(\mathrm{Im}(\varphi)) = \Phi(\tau_{n,m}) = M_{n,m} \in \mathbf{SO}(3). 
\end{eqnarray}

\noindent
Therefore, if $(x,y) \in \mathbb{R}^2$, then the matrix $M_{n,m}$  in $\mathbf{SO}(3)$ is given by
{\fontsize{10}{10}
\begin{eqnarray}
    \begin{pmatrix}\label{Matrix:Lawson}
    \cos(2y)                       &&  \sin(2y)\sin\big((m+n)x \big)                 & & -  \sin(2y)\cos\big((m+n)x\big)\\
    - \sin(2y)\sin\big((n-m)x\big) && \cos^2(y)\cos\big(2mx\big)+ \sin^2(y)\cos(2nx) & &  \cos^2(y)\sin(2mx)+ \sin^2(y)\sin(2nx)\\
    \sin(2y) \cos\big((n-m)x\big)  && -\cos^2(y)\sin(2mx)+ \sin^2(y)\sin(2nx)        & & \cos^2(y)\cos\big(2mx\big) - \sin^2(y)\cos(2nx)
    \end{pmatrix}.\\ \nonumber
\end{eqnarray}
}

\noindent
Using the Euler transformation $\mathcal{E}: \mathbb{S}^3(2) \rightarrow \mathbf{SO}(3)$, we shall denote throughout the following:
{\fontsize{10}{10}
\begin{eqnarray}\label{Eq:ponto}
    p(x,y) &=& \bigg(\cos(2y), \,  -\sin(2y)\sin\big(n-m)x\big), \, \sin(2y)\cos\big((n-m)x\big)\bigg),\\ \nonumber
    V(p) &=& \bigg( \sin(2y)\sin\big((m+n)x \big), \,   \cos^2(y)\cos\big(2mx\big)+ \sin^2(y)\cos(2nx), \, -\cos^2(y)\sin(2mx)+ \sin^2(y)\sin(2nx) \bigg).\\ \label{Eq:vetor}
\end{eqnarray}

\noindent

\noindent Observe that $(x, y) \longmapsto  p(x,y)=\big(\cos(2y), \,  -\sin(2y)\sin\big(n-m)x\big), \, \sin(2y)\cos\big((n-m)x\big)\big)$ is a parametrization of the unit $2$-sphere $\mathbb{S}^2$, where $y$ is the latitude and $x$ is the longitude, for integers $n$ and $m$ such that $(n,m)=1$. 

\noindent Similarly, the classes of compact minimal surfaces of Euler characteristic zero in $\mathbb{S}^3(2)$, given by $\tau_{n,n+1}$, have the explicit formula
\begin{eqnarray}\label{Eq:Lawson_tau_n+1 n}
    \varphi_{n,n+1}(x, y)&=& 
    \bigg(2\cos\big((n+1)x\big)\cos(y),\,
    2\sin\big((n+1)x\big) \cos(y),\, 
    2\cos(nx) \sin(y),\, 
    2\sin(nx) \sin(y)\bigg),      
\end{eqnarray}

\noindent which corresponds to the matrix $M_{n,n+1}$ in $\mathbf{SO}(3)$, described by
{\fontsize{10}{10}
\begin{eqnarray}\label{Matrix:Lawson_n_n+1}
    \begin{pmatrix}
    \cos(2y)        &&  \sin(2y) \sin\big((2n+1)x\big)                    & &  - \sin(2y)\cos\big((2n+1)x\big)\\
    \sin(2y)\sin(x) && \cos^2(y)\cos\big(2(n+1)x\big)+ \sin^2(y)\cos(2nx) & &  \cos^2(y)\sin\big(2(n+1)x\big)+ \sin^2(y)\sin(2nx)\\
    \sin(2y)\cos(x) && -\cos^2(y)\sin\big(2(n+1)x\big)+ \sin^2(y)\sin(2nx) & & \cos^2(y)\cos\big(2(n+1)x\big) - \sin^2(y)\cos(2nx)
    \end{pmatrix},
\end{eqnarray}}

\noindent and the classes of compact minimal surface of Euler characteristic zero in $\mathbb{S}^3(2)$ given by $\tau_{2n+1,2n+3}$ have the explicit formula
{\fontsize{10}{10}\selectfont
\begin{eqnarray}\label{Eq:Lawson_tau_2n+1 2n+3}
    \varphi_{2n+1,2n+3}(x, y) = 
    \bigg(2\cos\big((2n+1)x\big)\cos(y),\,
    2\sin\big((2n+1)x\big) \cos(y),\, 
    2\cos((2n+3)x) \sin(y),\, 
    2\sin((2n+3)x) \sin(y)\bigg),        
    \end{eqnarray}}
    which corresponds to the matrix $M_{2n+1,2n+3} \in \mathbf{SO}(3)$, described by
{\fontsize{9}{9}
\begin{align}
    \begin{pmatrix}\label{Matrix:Lawson_impares}
    \cos(2y)                       &&  \sin(2y)\sin\big(4nx \big)                 & & -  \sin(2y)\cos\big(4nx\big)\\
    \sin(2y)\sin\big(2x\big) && \cos^2(y)\cos\big(2(2n+1))x\big)+ \sin^2(y)\cos(2(2n+3))x) & &  \cos^2(y)\sin(2(2n+1)x)+ \sin^2(y)\sin(2(2n+3)x)\\
    \sin(2y) \cos\big(2x\big)  && -\cos^2(y)\sin(2(2n+1)x)+ \sin^2(y)\sin(2(2n+3)x)        & & \cos^2(y)\cos\big(2(2n+1)x\big) - \sin^2(y)\cos(2(2n+3)x)
    \end{pmatrix}. 
\end{align}
}    

\noindent We aim to determine the number of vector fields associated with point $p$, generated by two distinct parameter pairs. Now, for the Lawson surface $\tau_{n,n+1}$, we shall prove that there is only one such vector field $V(p) \in T^1\mathbb{S}^2$ at the point 
\begin{eqnarray}\label{Eq:ponto_n_n+1}  
    p(x,y) & = \big(\cos(2y), \,  \sin(2y)\sin\big(x\big), \, \sin(2y)\cos(x) \big),
\end{eqnarray}
    where
\begin{align}
  V(p) & = \bigg( \sin\big((2n+1)x\big), \,   \cos^2(y)\cos\big(2(n+1)x\big)+ \sin^2(y)\cos(2nx), \, -\cos^2(y)\sin\big(2(n+1)x\big)+ \sin^2(y)\sin(2nx) \bigg).
    \label{Eq:vetor_n_n+1}   
\end{align}

\begin{lemma} \label{Lemma:sucessivos}
    Consider the Lawson surface $\tau_{n,n+1}$, for each integer $n\geq 1$. If $A_1, A_2 \in \mathbb{R}^2$, $p_1 = p(A_1)$ and $p_2 = p(A_2)$, then
\begin{itemize}
    \item [i)] $p_1 = p_2$, if and only if, there exist $(a, b) \in \mathbb{R}^2$ such that
    \begin{eqnarray*}
        A_1=(
    a, b) \hspace{0.4cm} \mbox{and}\hspace{0.4cm} A_2 =(k_1\pi-b, a + (2k_2+1)\pi), \hspace{0.2cm}k_1, k_2 \in \mathbb{Z}.
    \end{eqnarray*}
    \item [ii)] If $p_1 = p_2$ then $V(p_1)=V(p_2)$.
\end{itemize}    
\end{lemma}
\begin{proof}

The result follows from a straightforward computation using the equations \eqref{Eq:ponto_n_n+1} and \eqref{Eq:vetor_n_n+1}, which are provided through the Euler transformation $\mathcal{E}: \mathbb{S}^3(2) \rightarrow \mathbf{SO}(3)$.
\end{proof}
\noindent
It follows immediately from the Lemma \ref{Lemma:sucessivos} that
\begin{proposition}\label{Prop:UnicoVetor}
    The Lawson surface $\tau_{n,n+1}$ defines, through the Euler transformation $\mathcal{E}: \mathbb{S}^3(2) \rightarrow \mathbf{SO}(3)$, only one unit vector field on $\mathbb{S}^2\backslash\{N,S\}$, for each integer $n\geq 1$.
\end{proposition}

Our next step will be to establish a relation between a compact minimal Lawson surface $\tau_{n,n+1}$ in $\mathbb{S}^3(2)$ and area-minimizing unit vector fields $V_k$ on $\mathbb{S}^2\backslash\{N,S\}$, as defined in \eqref{Eq:Vks}, where $k \in 2\mathbb{Z}\backslash\{0,2\}$ is Poincaré index around the singularity $N$ or $S$.


\begin{proposition}\label{Prop:ruled_Par} For each integer $n$, the compact minimal Lawson surface $\tau_{n,n+1}$ in $\mathbb{S}^3(2)$ corresponds to an area-minimizing unit vector field $V_k$ on $\mathbb{S}^2\backslash\{N,S\}$, where $k \in 2\mathbb{Z}\backslash\{0,2\}$ is the Poincaré index around the singularity at $N$ or $S$.
\end{proposition}

\begin{proof} 
A compact minimal Lawson surface $\tau_{n, n+1}$ of Euler characteristic zero in $\mathbb{S}^3(2)$ is given by \eqref{Eq:Lawson_tau_n+1 n}, which corresponds to the matrix $M_{n,n+1}$ in $\mathbf{SO}(3)$ described in \eqref{Matrix:Lawson_n_n+1}. The notation established through the equations \eqref{Eq:ponto} and \eqref{Eq:vetor} provide us 
{\fontsize{9,3}{9,3}
\begin{align*}
    p(x,y) &= \bigg(\cos(2y), \sin(2y)\sin(x), \sin(2y)\cos(x)\bigg), \\ 
    V(p) &= \bigg(\sin(2y) \sin\big((2n+1)x\big), \cos^2(y)\cos\big(2(n+1)x\big)+ \sin^2(y)\cos(2nx), -\cos^2(y)\sin\big(2(n+1)x\big)+ \sin^2(y)\sin(2nx) \bigg). 
\end{align*}}

\noindent
Consider $u_1(p) = (0, \cos(x), - \sin(x) )$ and $u_2(p) = (-\sin(2y), \cos(2y)\sin(x), \cos(2y)\cos(x))$ the unit vector fields at the point $p(x,y)$, where $u_1(p)$ is tangent to the parallel $\cos(2y)$ and $u_2(p)$ tangent to the meridian $y$.  We shall calculate the angle $\theta(p)$ between the unit vector fields $u_i(p)$ and $V(p)$, using the standard Euclidean inner product in $\mathbb{R}^3$, where $i=1,2$. It follows that, for the unit vector fields $u_2(p)$ and $V(p)$ we have
{\fontsize{10}{10}
\begin{align*}
    \cos\big(\theta(p)\big) &= -\sin^2(2y)\sin\big((2n+1)x\big) + \cos^2(y)\cos(2y)\bigg(\sin(x)\cos\big(2(n+1)x\big) - \cos(x)\sin\big(2(n+1)x\big)\bigg) \\
    & \hspace{3cm} + \sin^2(y)\cos(2y)\bigg(\sin(x)\cos(2nx)+\cos(x)\sin(2n x)\bigg)\\
    &= -\sin^2(2y)\sin\big((2n+1)x\big) + \cos^2(y)\cos(2y)\sin\big(x-2(n+1)x\big) 
    + \sin^2(y)\cos(2y)\sin\big(x+2nx)\big)\\
    &= -\sin^2(2y)\sin\big((2n+1)x\big) + \cos(2y)\sin\big((2n+1)x\big)\bigg(\cos^2(y)- \sin^2(y)\bigg)\\
    & = -\sin\big((2n+1)x\big)\bigg(\sin^2(2y)+\cos^2(2y)\bigg)\\
    & = \cos\bigg((2n+1)x + \frac{\pi}{2}\bigg).
\end{align*}}
Thus, the angle between the unit vector fields $u_2(p)$ and $V(p)$ at the point $p$ is given by
\begin{eqnarray*}
    \theta(p) = (2n+1)x + \frac{\pi}{2} +2k_1\pi, \, \, k_1 \in \mathbb{Z}.
 \end{eqnarray*}

\noindent Similarly, the angle between the unit vector fields $u_1(p)$ and $V(p)$ at the point $p$ is given by
\begin{eqnarray*}
    \theta(p) = (2n+1)x + 2k_1\pi, \, \, k_1 \in \mathbb{Z}.
 \end{eqnarray*}
One concludes that
\begin{eqnarray*}
d\theta_p\big(u_1(p)\big) = \frac{2n+1}{\cos(2y)} \hspace{0.4cm}\mbox{and} \hspace{0.4cm} d\theta_p\big(u_2(p)\big) \equiv 0,
\end{eqnarray*}
which means that $V$ makes constant angles with meridians, and the angle $\theta$ with parallel $\cos(2y)$, varies constantly. The equation \eqref{Eq:variacao_constante} implies that 
\begin{eqnarray*}
 k = 2n+2, \,  \, \, n \in \mathbb{Z}.
\end{eqnarray*}
Therefore, the unit vector field $V(p)$ is an area-minimizing vector field $V_k$ in $\mathbb{S}^2\backslash\{N,S\}$, where Poincaré index $k$ around the singularity $N$ or $S$  is a even positive integer, defined in~\eqref{Eq:Vks}.
\end{proof}

\noindent 

The Proposition \ref{Prop:ruled_Par} ensures that for each integer $n \in \mathbb{Z}$, the Lawson surface $\tau_{n,n+1}$ defines only one vector field on $\mathbb{S}^2\backslash\{N,S\}$. Proposition \ref{Prop:ruled_Impar} provides that the indices of these vector fields are even, different from zero and two, which immediately implies Theorem \ref{Thm:Main1}.



Our focus now is on compact minimal Lawson surfaces $\tau_{r,s}$ in $\mathbb{S}^3(2)$ such that $(r,s) = 1$, where $r=2n+1$ and $s=2n+3$, $n \in \mathbb{Z}$. We shall prove that, for each integer $n \in \mathbb{Z}$, a compact minimal Lawson surface $\tau_{2n+1,2n+3}$ corresponds to an area-minimizing unit vector field $V_k$ on $\mathbb{S}^2\backslash\{N,S\}$, as defined in \eqref{Eq:Vks}, where $k$ is an odd integer Poincaré index around the singularity $N$ or $S$.

\begin{proposition}\label{Prop:ruled_Impar} For each integer $n$, the compact minimal Lawson surface $\tau_{2n+1,2n+3}$ in $\mathbb{S}^3(2)$ corresponds to an area-minimizing unit vector field $V_k$ on $\mathbb{S}^2\backslash\{N,S\}$, where $k$ is the Poincaré index around the singularity at $N$ or $S$ and $k$ is an odd integer.
\end{proposition} 

\begin{proof} A compact minimal Lawson surface $\tau_{2n+1,2n+3}$ of Euler characteristic zero in $\mathbb{S}^3(2)$ is given by \eqref{Eq:Lawson_tau_2n+1 2n+3}, which corresponds to the matrix $M_{2n+1,2n+3}$ in $\mathbf{SO}(3)$ described in \eqref{Matrix:Lawson_impares}. Let $p$ be a point in $\mathbb{S}^2$ and  let $V(p)$ a unit vector field in $T^1\mathbb{S}^2$ given by 
\begin{eqnarray}\label{Eq:ponto_impares}
    p(x,y) &=& \bigg(\cos(2y), \,  \sin(2y)\sin\big(2x\big), \, \sin(2y)\cos\big(2x\big)\bigg) \in \mathbb{S}^2 ,\\ 
    V(p) &=& \bigg( v_1(p), \,  v_2(p), \, v_3(p) \bigg) \in T^1\mathbb{S}^2, \label{Eq:vetor_impares}
\end{eqnarray}
where \begin{eqnarray*}
    v_1(p) &=& \sin(2y)\sin\big(2(2n+1)x \big), \\
    v_2(p) &=& \cos^2(y)\cos\big(2(2n+1)x)\big) + \sin^2(y)\cos(2(2n+3)x),\\ 
    v_3(p) &=& -\cos^2(y)\sin(2(2n+1)x)+ \sin^2(y)\sin(2(2n+3)x).
\end{eqnarray*}

\noindent Consider a unit vector field  $u_1(p) = (0, \cos(2x), \sin(2x))$ tangent to the parallel $\cos(2y)$ and a unit vector field  $u_2(p) = (-\sin(2y), -\cos(2y)\sin(x), \cos(2y)\cos(x))$ tangent to the meridian $2x$ at the point $p(x,y)$. We shall calculate the angle $\theta(p)$ between the unit vector fields $u_i(p)$ and $V(p)$, using the standard Euclidean inner product in $\mathbb{R}^3$, where $i=1,2$. Taking the meridian $2y = \alpha$ and parallel $2x = \beta$. It follows that, for the unit vector fields $u_1(p)$ and $V(p)$ we get
{\fontsize{10}{10}
\begin{align*}
    \cos\big(\theta(p)\big) &= \cos^2\bigg(\frac{\alpha}{2} \bigg) \bigg( \cos(\beta)\cos\big((2n+1)\beta\big) - \sin(\beta)\sin\big((2n+1)\beta \bigg)   \\
    & \hspace{3cm} + \sin^2\bigg(\frac{\alpha}{2} \bigg) \bigg( \cos(\beta)\cos\big((2n+3)\beta\big) + \sin(\beta)\sin\big((2n+3)\beta \bigg) \\
    & = \cos^2\bigg(\frac{\alpha}{2} \bigg) \cos((2n+2)\beta) + \sin^2\bigg(\frac{\alpha}{2} \bigg)\cos(-(2n+2)\beta)\\
    & = \cos\big((2n+2)\beta\big).
\end{align*}}
Thus, the angle between the unit vector fields $u_1(p)$ and $V(p)$ at the point $p$ is given by
\begin{eqnarray*}
    \theta(p) = (2n+2)\beta + 2k_1\pi, \, \, k_1 \in \mathbb{Z}.
 \end{eqnarray*}

\noindent Similarly, the angle between the unit vector fields $u_2(p)$ and $V(p)$ at the point $p$ is given by
\begin{eqnarray*}
    \theta(p) = (2n+2)\beta +  \frac{\pi}{2} + 2k_1\pi, \, \, k_1 \in \mathbb{Z}.
 \end{eqnarray*}
One concludes that
\begin{eqnarray*}
d\theta_p\big(u_1(p)\big) = \frac{2n+2}{\cos(2\beta)} \hspace{0.4cm}\mbox{and} \hspace{0.4cm} d\theta_p\big(u_2(p)\big) \equiv 0,
\end{eqnarray*}
which means that $V$ makes constant angles with meridians, and the angle $\theta$ with parallel $\cos(2\beta)$, varies constantly.
The equation \eqref{Eq:variacao_constante} provides us 
\begin{eqnarray*}
 k = 2n+3, \,  \, \, n \in \mathbb{Z}.
\end{eqnarray*}
Therefore, the unit vector field $V(p)$ is an area-minimizing vector field $V_k$ in $\mathbb{S}^2\backslash\{N,S\}$ where Poincaré index $k$ around the singularity $N$ or $S$ is a odd positive integer.    
\end{proof}

Unlike the case of the Lawson surface $\tau_{n, n+1}$, we shall prove that, for the Lawson surface $\tau_{r,s}$ where $r=2n+1$ and $s=2n+3$, there is more than one vector field $V(p) \in T^1\mathbb{S}^2$ associated with the point $p$ generated by two distinct pairs of parameters. Our aim is to determine how many of these vector fields exist and to identify each specific vector field.

Consider $D$ a subset of the $2$-dimensional Euclidean space $\mathbb{R}^2$ defined as the square given by $[-\pi,\pi]\times[-\pi,\pi]$, and observe that
\begin{eqnarray*}
    \varphi_{n,m}(D) = \tau_{n,m},
\end{eqnarray*}
where $\varphi_{n,m}$ is the doubly periodic immersion defined in~\eqref{Eq:tau_Lawson_3-sphere(2)} and $m,n$ is a pair of integers with $(m,n)=1$.

\vspace{0.5cm}

\noindent 
Consider the intervals in $\mathbb{R}$ given by
\begin{eqnarray}\label{Eq:SubSquare}
D_1 = \left[0, \dfrac{\pi}{2}\right], \,\, \,D_2 = \left[\dfrac{\pi}{2}, \pi \right], \,\,\, D_3 = \left[-\pi , -\dfrac{\pi}{2}\right]\,\,\, \mbox{and} \,\,\, D_4 = \left[-\dfrac{\pi}{2}, 0\right].    
\end{eqnarray}
Therefore,
\begin{eqnarray}\label{Eq:Square}
    D = \bigcup_{i,j=1}^4 (D_i\times D_j).
\end{eqnarray}

\noindent 
We note that the domain $D$ covers the compact minimal Lawson surface $\tau_{2n+1,2n+3}$ twice, while the Proposition \ref{Prop:DomainG} provides a domain $G$ for the doubly periodic immersion $\varphi_{2n+1,2n+3}$, ensuring that G covers $\tau_{2n+1,2n+3}$ exactly once.



\begin{proposition}\label{Prop:DomainG}
    If $G=[-\pi,\pi]\times\left[-\dfrac{\pi}{2},\dfrac{\pi}{2}\right]$ then $\varphi_{r,s}(G) = \varphi_{r,s}(D)$, where $r=2n+1$ and $s=2 n+3$, $n \in \mathbb{Z}$.
\end{proposition}
\begin{proof}
Consider the intervals in $\mathbb{R}$ given by
\begin{eqnarray*}
D_1' = D_3, \hspace{0.5cm} D_2' = D_4\hspace{0.5cm} D_3' = D_1, \hspace{0.5cm}\mbox{and}\hspace{0.5cm}\ D_4' = D_2,   
\end{eqnarray*}
where the subset $D_i$ of $D$ is defined in \eqref{Eq:SubSquare}, for all $1 \leq i \leq 4$. By straightforward computation one obtains
\begin{eqnarray*}
   \varphi_{r,s}\left( D_i \times D_j\right) = \varphi_{r,s}\left( D_i' \times D_j'\right)  
\end{eqnarray*}
That implies
\begin{eqnarray*}
    \varphi_{r,s}\left( \bigcup_{i=1}^4 D_i \times D_1 \right) =  \varphi_{r,s} \left(\bigcup_{i=1}^4D_i \times D_3\right)\,\,\, \mbox{and} \,\, \, \,
    \varphi_{r,s}\left( \bigcup_{i=1}^4 D_i \times D_2 \right) =  \varphi_{r,s} \left(\bigcup_{i=1}^4D_i \times D_4\right).\\
\end{eqnarray*}

\noindent
Observe that, $G$ can be written as
\begin{eqnarray*}
    G = \left( \bigcup_{i=1}^4 D_i \times D_1 \right)\bigcup \left(\bigcup_{i=1}^4D_i \times D_4\right).
\end{eqnarray*}
Therefore, $\varphi_{r,s}(G) = \varphi_{r,s}(D)$, where $r=2n+1$ and $s=2 n+3$, $n \in \mathbb{Z}$.
\end{proof}

\begin{proposition}\label{Prop:Dois_Vetores-impares}
    Let $G_1=[-\pi,\pi]\times\left[0,\frac{\pi}{2}\right]$ and $G_2=[-\pi,\pi]\times\left[-\frac{\pi}{2},0\right]$ be the indicated subsets of $G$. Consider $\mathcal{E}: \mathbb{S}^3(2) \rightarrow \mathbf{SO}(3)$ the Euler transformation and $\varphi_{r,s}$ defined in \eqref{Eq:Lawson_tau_2n+1 2n+3}, with $r=2n+1$ and $s=2 n+3$, $n \in \mathbb{Z}$. If $(x_1,y_1)$ and $(x_2,y_2)$ belong to $\overset{o}{G_1}\cup \overset{o}{G_2}$  are such that
\begin{itemize}
    \item [i)] $\varphi_{r,s}(x_1,y_1) = p_1$, $\varphi_{r,s}(x_2,y_2)=p_2$;
    \item [ii)] $\mathcal{E}\big(\varphi_{r,s}(x_1,y_1)\big) = \big(p_1, V(p_1)\big)$ and $\mathcal{E}\big(\varphi_{r,s}(x_2,y_2)\big) = \big(p_2, V(p_2)\big)$.
\end{itemize}
Then,
\begin{eqnarray*}
    V(\varphi_{r,s}(x_1,y_1)) = - V(\varphi_{r,s}(x_2,y_2)).
\end{eqnarray*}
\end{proposition}

\begin{proof}
Set $A_1 := (x_1,y_1)$ and $A_2: =(x_2, y_2)$ belong to $\overset{o}{G_1}\cup \overset{o}{G_2}$. Since $A_1$ and $A_2$ are such that $\varphi_{r,s}(A_1) = \varphi_{r,s}(A_2)$, the Euler transformation $\mathcal{E}: \mathbb{S}^3(2) \rightarrow \mathbf{SO}(3)$ through the equation \eqref{Eq:ponto_impares} provides us  
\begin{eqnarray} \label{Eq:A_2-impares}
A_2 = \left(\, \dfrac{(2k_1+1)\pi}{2} + x_1,\, \, k_2\pi-x_2\right), \hspace{0.2cm}k_1, k_2 \in \mathbb{Z}. 
\end{eqnarray}    

\noindent 
In order to verify the equality $V(\varphi_{r,s}(x_1,y_1)) = - V(\varphi_{r,s}(x_2,y_2))$, it is sufficient to compare equation \eqref{Eq:vetor_impares} using the parameters $A_1 = (x_1,y_1)$ and $A_2$ obtained in \eqref{Eq:A_2-impares}.
\end{proof}

\begin{proof} \textit{of Theorem \ref{Thm:Main2}}.
    The topological closure of the image of area-minimizing unit vector field $V_k$ on $\mathbb{S}^2\backslash\{N,S\}$ is a minimally immersed Klein bottle in $T^1\mathbb{S}^2$, as long as the Poincaré index $k$ around the singularity $N$ or $S$ is an even positive integer, as mentioned in subsection \ref{GK}. The Proposition \ref{Prop:UnicoVetor} assures us that, via Euler transformation $\mathcal{E}: \mathbb{S}^3(2) \rightarrow \mathbf{SO}(3)$, for each $p \in \mathbb{S}^3(2)$, there is only one vector $V_k(p) \in T^1_p \mathbb{S}^2$. Furthermore, the Proposition \ref{Prop:ruled_Par} provides us a correspondence between Lawson surfaces of the type $\tau_{n,n+1}$ and area-minimizing vector fields $V_k$, where $k \in 2\mathbb{Z}\backslash\{0,2\}$. Therefore, the classes of compact minimal surfaces of Euler characteristic zero in $\mathbb{S}^3(2)$ given by $\tau_{n,n+1}$ are exactly minimally immersion Klein bottles.

\noindent 
 Consider $r= 2n+1$ and $s=2n+3$, $n \in \mathbb{Z}$, such that $(r,s)=1$. There exists a correspondence between Lawson surfaces $\tau_{r,s}$ and area-minimizing vector fields $V_k$, where the Poincaré index $k$ around the singularity $N$ or $S$ is an odd integer, obtained by Proposition \ref{Prop:ruled_Impar}.  From Proposition \ref{Prop:Dois_Vetores-impares}, we have that if there exist two pairs of parameters $(x_1,y_1)$ and $(x_2,y_2)$ such that both are distinct and generate the same point $p \in \mathbb{S}^3(2)$, then via the Euler transformation $\mathcal{E}: \mathbb{S}^3(2) \rightarrow \mathbf{SO}(3)$, we obtain two vectors in $\mathbb{S}^2$, $V_k(p)$ and $-V_k(p)$. Consider the subsets of $\tau_{2n+1,2n+3}$ given by
\begin{eqnarray}\label{Eq:Lawson_Cylinders}
C_{rs}^+ = \varphi_{r,s}\left([-\pi , \pi]\times \left[0, \dfrac{\pi}{2}\right]\right)\hspace{0.2cm} \mbox{and} \hspace{0.2cm} C_{rs}^- = \varphi_{r,s}\left( \left[-\pi , \pi\right]\times \left[-\dfrac{\pi}{2},0\right]\right).    
\end{eqnarray}
Observe that $C_{rs}^+$ and $C_{rs}^-$ are cylindrical components of $\tau_{r,s}$, meaning that $C_{rs}^+\cup C_{rs}^- = \tau_{r,s}$. Through the Euler transformation $\mathcal{E}: \mathbb{S}^3(2) \rightarrow \mathbf{SO}(3)$, there is a correspondence between $C_{rs}^+$ and $C_{rs}^-$ with the area-minimizing vector fields $V_k$ and $-V_k$, respectively. 

\end{proof}


\section*{References}

\begin{itemize}
    \item[\textbf{[1]}] Vincent Borrelli and Olga Gil-Medrano. Area-minimizing vector fields on round 2-spheres. \textit{J. Reine Angew. Math.}, 640:85–99, 2010.

    \item[\textbf{[2]}] Fabiano G. B. Brito, Jackeline Conrado, Icaro Gonçalves, Adriana V. Nicoli, and Giovanni Nunes. Minimally Immersed Klein Bottles in the Unit Tangent Bundle of the Unit 2-Sphere Arising from Area-Minimizing Unit Vector Fields on \( S^2 \setminus \{N,S\} \). \textit{J. Geom. Anal.}, 33(5):142, 2023.

    \item[\textbf{[3]}] Fabiano G. B. Brito, Jackeline Conrado, Icaro Gonçalves, and Adriana V. Nicoli. Area minimizing unit vector fields on antipodally punctured unit 2-sphere. \textit{Comptes Rendus Mathématique}, 359-10:1225–1232, 2021.

    \item[\textbf{[4]}] Jackeline Conrado. Minimally immersed surfaces in the unit tangent bundle of the 2-sphere arising from area-minimizing unit vector fields on \( S^2 \setminus \{N,S\} \), \textit{PhD thesis of the University of São Paulo. PhD thesis of University of São Paulo}, 2022.

    \item[\textbf{[5]}] Olga Gil-Medrano and Elisa Llinares-Fuster. Minimal unit vector fields. \textit{Tohoku Math. J. (2)}, 54(1):71–84, 2002.

    \item[\textbf{[6]}] Herman Gluck and Wolfgang Ziller. On the volume of a unit vector field on the three-sphere. \textit{Comment. Math. Helv.}, 61(2):177–192, 1986.

    \item[\textbf{[7]}] David L. Johnson. Volumes of flows. \textit{Proc. Amer. Math. Soc.}, 104(3):923–931, 1988.

    \item[\textbf{[8]}] Wilhelm Klingenberg and Shigeo Sasaki. On the tangent sphere bundle of a 2-sphere. \textit{Tohoku Math. J. (2)}, 27:49–56, 1975.

    \item[\textbf{[9]}] H. Blaine Lawson, Jr. Complete minimal surfaces in \( S^3 \). \textit{Ann. of Math. (2)}, 92:335–374, 1970.

    \item[\textbf{[10]}] Antonio Ros. The willmore conjecture in the real projective space. \textit{Mathematical Research Letters}, 6, 05 2000.
\end{itemize}

\end{document}